\documentclass{amsart}

\usepackage{amsmath}
\usepackage{amssymb}
\usepackage{xcolor}

\newtheorem{thm}{Theorem}[section]

\newtheorem{lem}[thm]{Lemma}
\newtheorem{propo}[thm]{Proposition}

\theoremstyle{definition}

\theoremstyle{remark}

\title[The Riesz transform]{Discrete harmonic analysis \\associated with Jacobi expansions II: \\the Riesz transform}
\author[A. Arenas]{Alberto Arenas}
\address{Departamento de Matem\'aticas y Computaci\'on,
Universidad de La Rioja, Complejo Cient\'{\i}fico-Tecnol\'ogico,
Calle Madre de Dios 53, 26006 Logro\~no, Spain}
\email{alberto.arenas@unirioja.es}

\author[\'O. Ciaurri]{\'Oscar Ciaurri}
\address{Departamento de Matem\'aticas y Computaci\'on,
Universidad de La Rioja, Complejo Cient\'{\i}fico-Tecnol\'ogico,
Calle MAdre de Dios 53, 26006 Logro\~no, Spain}
\email{oscar.ciaurri@unirioja.es}

\author[E. Labarga]{Edgar Labarga}
\address{Departamento de Matem\'aticas y Computaci\'on,
Universidad de La Rioja, Complejo Cient\'{\i}fico-Tecnol\'ogico,
Calle Madre de Dios 53, 26006 Logro\~no, Spain}
\email{edgar.labarga@unirioja.es}

\thanks{The first-named author was supported by a predoctoral research grant of the Government of Comunidad Aut\'{o}noma de La Rioja. The second-named author was supported by grant MTM2015-65888-C04-4-P MINECO/FEDER, UE, from Spanish Government. The third-named author was supported by a predoctoral research grant of the University of La Rioja.}

\begin{document}
%%%%%%%%%%%%%%%%%%%%%%%%%%%%%%%%%%%%%%%%%%%%%%%%%%%

%%%%%%%%%%%%%%%%%%%%%%%%%%%%%%%%%%%%%%%%%%%%%%%%%%%%%%
\begin{abstract}
This paper is the continuation of the study on discrete harmonic analysis related to Jacobi expansions initiated in \cite{ACL-JacI}. Considering the operator $\mathcal{J}^{(\alpha,\beta)}=J^{(\alpha,\beta)}-I$, where $J^{(\alpha,\beta)}$ is the three-term recurrence relation for the normalized Jacobi polynomials and $I$ is the identity operator, we focus on the study of weighted inequalities for the Riesz transform associated with it.
\end{abstract}
%%%%%%%%%%%%%%%%%%%%%%%%%%%%%%%%%%%%%%%%%%%%%%%%%%%%%%

\maketitle

%%%%%%%%%%%%%%%%%%%%%%%%%%%%%%%%%%%%%%%%%%%%%%%%%%%%%%
\section{Introduction}
%%%%%%%%%%%%%%%%%%%%%%%%%%%%%%%%%%%%%%%%%%%%%%%%%%%%%%

For $\alpha,\beta>-1$ and $n=0,1,2,\dots$, we consider the sequences $\{a_{n}^{(\alpha,\beta)}\}_{n\in\mathbb{N}}$ and $\{b_{n}^{(\alpha,\beta)}\}_{n\in\mathbb{N}}$ given by
\[
a_n^{(\alpha,\beta)}=\frac{2}{2n+\alpha+\beta+2}
\sqrt{\frac{(n+1)(n+\alpha+1)(n+\beta+1)(n+\alpha+\beta+1)}{(2n+\alpha+\beta+1)(2n+\alpha+\beta+3)}},\quad n\geq 1,
\]
\[
a_{0}^{(\alpha,\beta)} = \frac{2}{\alpha+\beta+2}\sqrt{\frac{(\alpha+1)(\beta+1)}{(\alpha+\beta+3)}},
\]
\[
b_n^{(\alpha,\beta)}=\frac{\beta^2-\alpha^2}{(2n+\alpha+\beta)(2n+\alpha+\beta+2)},\quad n\geq 1,
\]
and
\[
b_{0}^{(\alpha,\beta)} = \frac{\beta-\alpha}{\alpha+\beta+2}.
\]
Then, for any given sequence $\{f(n)\}_{n\ge 0}$, we define $\{J^{(\alpha,\beta)}f(n)\}_{n\ge 0}$ by the relations
\[
J^{(\alpha,\beta)}f(n)=a_{n-1}^{(\alpha,\beta)}f(n-1)+b_n^{(\alpha,\beta)}f(n)+ a_{n}^{(\alpha,\beta)}f(n+1), \qquad n\ge 1,
\]
and $J^{(\alpha,\beta)}f(0)=b_0^{(\alpha,\beta)}f(0)+ a_{0}^{(\alpha,\beta)}f(1)$.

Note that the sequences $\{a_n^{(\alpha,\beta)}\}_{n\ge 0}$ and $\{b_n^{(\alpha,\beta)}\}_{n\ge 0}$ are the ones involved in the three-term recurrence relation for the normalized Jacobi polynomials. By using the Rodrigues' formula (see \cite[p.~67, eq.~(4.3.1)]{Szego}), the Jacobi polynomials $P^{(\alpha,\beta)}_n(x)$, $n\ge 0$, are defined as
\[
(1-x)^{\alpha}(1+x)^{\beta}P_n^{(\alpha,\beta)}(x)=\frac{(-1)^n}{2^n \, n!}\frac{d^n}{dx^n}\left((1-x)^{\alpha+n}(1+x)^{\beta+n}\right).
\]
They are orthogonal on the interval $[-1,1]$ with respect to the measure
\[
d\mu_{\alpha,\beta}(x)=(1-x)^\alpha(1+x)^{\beta}\,dx.
\]
The family $\{p_n^{(\alpha,\beta)}(x)\}_{n\ge 0}$, given by $p_n^{(\alpha,\beta)}(x)=w_n^{(\alpha,\beta)}P_n^{(\alpha,\beta)}(x)$, where
\begin{equation*}
\begin{aligned}
w_n^{(\alpha,\beta)}& = \frac{1}{\|P_n^{(\alpha,\beta)}\|_{L^2((-1,1),d\mu_{\alpha,\beta})}} \\&= \sqrt{\frac{(2n+\alpha+\beta+1)\, n!\,\Gamma(n+\alpha+\beta+1)}{2^{\alpha+\beta+1}\Gamma(n+\alpha+1)\,\Gamma(n+\beta+1)}},\quad n\geq1,
\end{aligned}
\end{equation*}
and
\[
w_{0}^{(\alpha,\beta)} = \frac{1}{\|P_{0}^{(\alpha,\beta)}\|_{L^2((-1,1),d\mu_{\alpha,\beta})}} = \sqrt{\frac{\Gamma(\alpha+\beta+2)}{2^{\alpha+\beta+1}\Gamma(\alpha+1)\Gamma(\beta+1)}},
\]
is a complete orthonormal system in the space $L^2([-1,1],d\mu_{\alpha,\beta})$. Furthermore, we have that
\[
J^{(\alpha,\beta)}p^{(\alpha,\beta)}_n(x)=xp_n^{(\alpha,\beta)}(x),\qquad x\in [-1,1].
\]

Throughout this paper we will work with the operator
\[
\mathcal{J}^{(\alpha,\beta)}f(n)=(J^{(\alpha,\beta)}-I)f(n),
\]
where $I$ denotes the identity operator, instead of $J^{(\alpha,\beta)}$. Due to this translation by the identity $I$, the operator  $-\mathcal{J}^{(\alpha,\beta)}$ is nonnegative and its spectrum is the interval $[0,2]$.

In this paper we continue the study of the discrete harmonic analysis associated with $\mathcal{J}^{(\alpha,\beta)}$ initiated in \cite{ACL-JacI}, where the heat semigroup was exhaustively analyzed. Our work on these kind of problems pretends to be an extension of the research done in   \cite{Ciau-et-al} for the discrete Laplacian
\begin{equation}
\label{eq:dis-lap}
\Delta_d f(n)=f(n-1)-2f(n)+f(n+1)
\end{equation}
and in \cite{Bet-et-al} for ultraspherical expansions, which corresponds with the case $\alpha=\beta=\lambda-1/2$ of $J^{(\alpha,\beta)}$.
Our target here is the study of a classical operator on harmonic analysis: the Riesz transform. For $\Delta_d$ this operator corresponds with classical discrete Hilbert transform and it was analyzed in \cite{Ciau-et-al}. For ultraspherical expansions this operator has not been treated yet, so our result is completely new even in that particular case.

Although the powers of $\mathcal{J}^{(\alpha,\beta)}$ will be studied deeply in a forthcoming paper, we have to state them at this point in order to define the Riesz transform. For our present purpose, it is enough to say that the fractional integrals (also known as negative powers) of $\mathcal{J}^{(\alpha,\beta)}$ are defined, for an appropriate sequence $\{f(n)\}_{n\ge 0}$, by
\[
(-\mathcal{J}^{(\alpha,\beta)})^{-\sigma}f(n)=\frac{1}{\Gamma(\sigma)}\int_{0}^{\infty}W_t^{(\alpha,\beta)}f(n)\,\frac{dt}{t^{1-\sigma}}, \qquad \sigma>0,
\]
where
\[
W_t^{(\alpha,\beta)}f(n)=\sum_{m=0}^\infty f(m)K_t^{(\alpha,\beta)}(m,n),
\]
is the heat semigroup associated to $\mathcal{J}^{(\alpha,\beta)}$ (see \cite{ACL-JacI}), whose kernel is
\[
K_t^{(\alpha,\beta)}(m,n)=\int_{-1}^{1}e^{-(1-x)t}p_m^{(\alpha,\beta)}(x)p_n^{(\alpha,\beta)}(x)\, d\mu_{\alpha,\beta}(x).
\]
As we will show in Proposition \ref{prop:well-def} below, for $\alpha,\beta\geq -1/2$, the operator $(-\mathcal{J}^{(\alpha,\beta)})^{-\sigma}$ is only well defined for $0<\sigma <1/2$.

We have that (see \cite[Section 3.1]{ACL-JacI})
\begin{equation*}
\mathcal{J}^{(\alpha,\beta)} = -\delta^{\star}\delta,
\end{equation*}
where
\begin{equation*}
\delta f(n) = d_{n}f(n) - e_{n}f(n+1),\quad n\geq 0,
\end{equation*}
\begin{equation*}
\delta^{\star} f(n) = d_n f(n) - e_{n-1}f(n-1),\quad n\geq 1,
\end{equation*}
and $\delta^{\star} f(0) = d_{0}f(0)$,
with the sequences $\{d_n\}_{n\ge 0}$ and $\{e_n\}_{n\ge 0}$ defined by $d_{0} = \sqrt{\frac{2(\alpha+1)}{\alpha+\beta+2}}$,
\begin{equation*}
d_{n} = \sqrt{\frac{2(n+\alpha+\beta+1)(n+\alpha+1)}{(2n+\alpha+\beta+1)(2n+\alpha+\beta+2)}},\quad n\geq 1,
\end{equation*}
and
\begin{equation*}
e_{n} = \sqrt{\frac{2(n+\beta+1)(n+1)}{(2n+\alpha+\beta+2)(2n+\alpha+\beta+3)}}\quad n\geq 0.
\end{equation*}
Note that, $\delta$ and $\delta^{\star}$ are adjoint operators in $\ell^2(\mathbb{N})$.

Following a standard procedure, for a given sequence $\{f(n)\}_{n\ge 0}$, the Riesz transform should be defined via composition by $\delta (-\mathcal{J}^{(\alpha,\beta)})^{-1/2}f(n)$. Unfortunately, this procedure does not work in our case because the operator $(-\mathcal{J}^{(\alpha,\beta)})^{-1/2}$ is not well defined so we need an alternative way to define the Riesz transform. In our situation, this operator is given by
\begin{equation}
\label{eq:Riesz-def}
\mathcal{R}f(n)=\lim_{\sigma\to \frac{1}{2}^{-}}\delta (-\mathcal{J}^{(\alpha,\beta)})^{-\sigma}f(n).
\end{equation}
This is a natural way to proceed and, in fact, it was used in \cite{Ciau-et-al} to define the Riesz transform for the discrete Laplacian \eqref{eq:dis-lap}. 

The Riesz transform is a classical operator in harmonic analysis and it has been analyzed in several settings. For example, the conjugate function and the Hilbert transform are the Riesz transform for the trigonometric Fourier series and for the one-dimensional Fourier transform, respectively, and both of them were analyzed by M. Riesz in his celebrated paper \cite{Riesz}. In the case of the $n$-dimensional Fourier transform the multiplier $\text{p.\,v.\,}\frac{x_j}{|x|^{n+1}}$ defines the $j$-th Riesz transform and such one is a prototype of singular integral. For non-trigonometric Fourier expansions this operator has been studied in many situations (see \cite{N-S} and the references therein). The Riesz transform has also been treated in very abstract settings as for example Riemannian manifolds or compact Lie groups.

In the main result of this paper we prove some weighted inequalities for $\mathcal{R}$. Before stated it, we need some preliminaries. A weight on $\mathbb{N}$ will be a strictly positive sequence $w=\{w(n)\}_{n\ge 0}$. We consider the weighted $\ell^{p}$-spaces
\[
\ell^p(\mathbb{N},w)=\left\{f=\{f(n)\}_{n\ge 0}: \|f\|_{\ell^{p}(\mathbb{N},w)}:=\Bigg(\sum_{m=0}^{\infty}|f(m)|^p w(m)\Bigg)^{1/p}<\infty\right\},
\]
$1\le p<\infty$, and the weak weighted $\ell^{1}$-space
\[
\ell^{1,\infty}(\mathbb{N},w)=\left\{f=\{f(n)\}_{n\ge 0}: \|f\|_{\ell^{1,\infty}(\mathbb{N},w)}:=\sup_{t>0}t\sum_{\{m\in \mathbb{N}: |f(m)|>t\}} w(m)<\infty\right\},
\]
and we simply write $\ell^p(\mathbb{N})$ and $\ell^{1,\infty}(\mathbb{N})$ when $w(n)=1$ for all $n\in \mathbb{N}$.

Furthermore, we say that a weight $w(n)$ belongs to the discrete Muckenhoupt $A_p(\mathbb{N})$ when
\[
\sup_{\begin{smallmatrix} 0\le n \le m \\ n,m\in \mathbb{N} \end{smallmatrix}} \frac{1}{(m-n+1)^p}\Bigg(\sum_{k=n}^mw(k)\Bigg)\Bigg(\sum_{k=n}^mw(k)^{-1/(p-1)}\Bigg)^{p-1} <\infty,
\]
for $1<p<\infty$,
\[
\sup_{\begin{smallmatrix} 0\le n \le m \\ n,m\in \mathbb{N} \end{smallmatrix}} \frac{1}{m-n+1}\Bigg(\sum_{k=n}^mw(k)\Bigg)\max_{n\le k \le m}w(k)^{-1} <\infty,
\]
for $p=1$.

\begin{thm}
\label{th:main}
Let $\alpha,\beta\ge -1/2$ and let $\mathcal{R}$ be the Riesz transform defined in \eqref{eq:Riesz-def}.
\begin{enumerate}
\item[(a)]
If $1<p<\infty$ and $w\in A_p(\mathbb{N})$, then
\begin{equation*}
%\label{eq:bound-trans}
\|\mathcal{R}f\|_{\ell^p(\mathbb{N},w)}\le C \|f\|_{\ell^p(\mathbb{N},w)}, \qquad f\in\ell^2(\mathbb{N})\cap\ell^{p}(\mathbb{N},w),
\end{equation*}
where $C$ is a constant independent of $f$. Consequently, the operator $\mathcal{R}$ extends uniquely to a bounded linear operator from $\ell^p(\mathbb{N},w)$ into itself.
\item[(b)]
If $w\in A_1(\mathbb{N})$, then
\begin{equation*}
%\label{eq:bound-trans-weak}
\|\mathcal{R}f\|_{\ell^{1,\infty}(\mathbb{N},w)}\le C \|f\|_{\ell^1(\mathbb{N},w)}, \qquad f\in \ell^2(\mathbb{N})\cap\ell^{1}(\mathbb{N},w),
\end{equation*}
where $C$ is a constant independent of $f$. Consequently, the operator $\mathcal{R}$ extends uniquely to a bounded linear operator from $\ell^1(\mathbb{N},w)$ into $\ell^{1,\infty}(\mathbb{N},w)$.
\end{enumerate}
\end{thm}

The paper is organized as follows. The proof of Theorem \ref{th:main} will be a consequence of a discrete Calder\'on-Zygmund theory which is given in the next section. In Section \ref{sec:frac-int} we show that, effectively, the fractional integrals $(-\mathcal{J}^{(\alpha,\beta)})^{-\sigma}$ are only well defined for $0<\sigma<1/2$. Section \ref{sec:proof} contains the proof of Theorem \ref{th:main} and Section \ref{sec:estimates} is focused on the proof of the main estimates to apply Calder\'on-Zygmund theory. In the last section some technical results used along the paper are proved.

%%%%%%%%%%%%%%%%%%%%%%%%%%%%%%%%%%%%%%%%%%%%%%%%%%%%%%%%%%%%%%%%%%%%%%
\section{Local theory for discrete Banach space valued Calder\'on-Zygmund operators}
%%%%%%%%%%%%%%%%%%%%%%%%%%%%%%%%%%%%%%%%%%%%%%%%%%%%%%%%%%%%%%%%%%%%%%

As we have already mentioned, the proof of Theorem \ref{th:main} relies on an appropriate local theory for discrete Banach space valued Calder\'on-Zygmund operators which is presented in \cite{Bet-et-al}. For the reader's convenience, it is appropriate to recall some of the basic aspects of this local theory.

Suppose that $\mathbb{B}_1$ and $\mathbb{B}_2$ are Banach spaces. We denote by $\mathcal{L}(\mathbb{B}_1,\mathbb{B}_2)$ the space of bounded linear operators from $\mathbb{B}_1$ into $\mathbb{B}_2$. Let us suppose that
\[
K:(\mathbb{N}\times\mathbb{N})\setminus D \longrightarrow \mathcal{L}(\mathbb{B}_1,\mathbb{B}_2),
\]
where $D:=\{(n,n):n\in \mathbb{N}\}$, is measurable and that for certain positive constant $C$ and for each $n$, $m\in \mathbb{N}$, the following conditions hold.
\begin{enumerate}
\item[(a)] The size condition:
\[
\|K(n,m)\|_{\mathcal{L}(\mathbb{B}_1,\mathbb{B}_2)}\le \frac{C}{|n-m|},
\]
\item[(b)] the regularity properties:
\begin{enumerate}
\item[(b1)]
\[
\|K(n,m)-K(l,m)\|_{\mathcal{L}(\mathbb{B}_1,\mathbb{B}_2)}\le C \frac{|n-l|}{|n-m|^2},\quad |n-m|>2|n-l|, \frac{m}{2}\le n,l\le \frac{3m}{2},
\]
\item[(b2)]
\[
\|K(m,n)-K(m,l)\|_{\mathcal{L}(\mathbb{B}_1,\mathbb{B}_2)}\le C \frac{|n-l|}{|n-m|^2},\quad |n-m|>2|n-l|, \frac{m}{2}\le n,l\le \frac{3m}{2}.
\]
\end{enumerate}
\end{enumerate}
A kernel $K$ satisfying conditions (a) and (b) is called a local $\mathcal{L}(\mathbb{B}_1,\mathbb{B}_2)$-standard kernel. For a Banach space $\mathbb{B}$ and a weight $w=\{w(n)\}_{n\ge 0}$, we consider the space
\[
\ell^{p}_{\mathbb{B}}(\mathbb{N},w)=\left\{ \text{$\mathbb{B}$-valued sequences } f=\{f(n)\}_{n\ge 0}: \{\|f(n)\|_{\mathbb{B}}\}_{n\ge 0}\in \ell^p(\mathbb{N},w)\right\}
\]
for $1\le p<\infty$, and
\[
\ell^{1,\infty}_{\mathbb{B}}(\mathbb{N},w)=\left\{ \text{$\mathbb{B}$-valued sequences } f=\{f(n)\}_{n\ge 0}: \{\|f(n)\|_{\mathbb{B}}\}_{n\ge 0}\in \ell^{1,\infty}(\mathbb{N},w)\right\}.
\]
As usual, we simply write $\ell_{\mathbb{B}}^r(\mathbb{N})$ and $\ell^{1,\infty}_{\mathbb{B}}(\mathbb{N})$ when $w(n)=1$ for all $n\in \mathbb{N}$. Also, by $\mathbb{B}_0^{\mathbb{N}}$ we represent the space of $\mathbb{B}$-valued sequences $f=\{f(n)\}_{n\ge 0}$ such that $f(n)=0$, with $n>j$, for some $j\in \mathbb{N}$.
\begin{thm}[Theorem 2.1 in \cite{Bet-et-al}]
\label{thm:CZ}
Let $\mathbb{B}_1$ and $\mathbb{B}_2$ be Banach spaces. Suppose that $T$
is a linear and bounded operator from $\ell_{\mathbb{B}_1}^r(\mathbb{N})$ into $\ell_{\mathbb{B}_2}^r(\mathbb{N})$, for some $1<r<\infty$, and such that there exists a local $\mathcal{L}(\mathbb{B}_1,\mathbb{B}_2)$-standard kernel $K$ such that, for every sequence $f\in (\mathbb{B}_1)_0^{\mathbb{N}}$,
\[
Tf(n)=\sum_{m=0}^{\infty}K(n,m)\cdot f(m),
\]
for every $n\in \mathbb{N}$ such that $f(n)=0$. Then,
\begin{enumerate}
\item[(i)] for every $1< p <\infty$ and $w\in A_p(\mathbb{N})$ the operator $T$
can be extended from $\ell_{\mathbb{B}_1}^r(\mathbb{N})\cap \ell_{\mathbb{B}_1}^p(\mathbb{N},w)$
to $\ell_{\mathbb{B}_1}^p(\mathbb{N},w)$ as a bounded operator from $\ell_{\mathbb{B}_1}^p(\mathbb{N},w)$ into $\ell_{\mathbb{B}_2}^p(\mathbb{N},w)$; 

\item[(ii)] for every $w\in A_1(\mathbb{N})$ the operator $T$
can be extended from $\ell_{\mathbb{B}_1}^r(\mathbb{N})\cap \ell_{\mathbb{B}_1}^1(\mathbb{N},w)$
to $\ell_{\mathbb{B}_1}^1(\mathbb{N},w)$ as a bounded operator from $\ell_{\mathbb{B}_1}^p(\mathbb{N},w)$ into $\ell_{\mathbb{B}_2}^{1,\infty}(\mathbb{N},w)$.
\end{enumerate}
\end{thm}

%%%%%%%%%%%%%%%%%%%%%%%%%%%%%%%%%%%%%%%%%%%%%%%%%%%%%%%%%
\section{The fractional integrals $(-\mathcal{J}^{(\alpha,\beta)})^{-\sigma}$}
\label{sec:frac-int}
%%%%%%%%%%%%%%%%%%%%%%%%%%%%%%%%%%%%%%%%%%%%%%%%%%%%%%%%%
As we have commented in the introduction, in this section we will show that $(-\mathcal{J}^{(\alpha,\beta)})^{-\sigma}$ can only be defined for $0<\sigma <1/2$.

In the following proposition we will use by the first time an estimate for the Jacobi polynomials that will be used frequently along the paper (see \cite[eq.~(2.6) and (2.7)]{Muckenhoupt}). If $-1<x<1$, $a,b>-1$, the estimate
\begin{multline}\label{eq:unif-bound-trozos}
  |p_n^{(a,b)}(x)|\\\le C \begin{cases}
                          (n+1)^{a+1/2}, & 1-1/(n+1)^{2}<x<1, \\
                          (1-x)^{-a/2-1/4}(1+x)^{-b/2-1/4}, & -1+1/(n+1)^{2}\leq x\leq 1-1/(n+1)^{2},\\
                          (n+1)^{b+1/2}, & -1<x<-1+1/(n+1)^{2},
                        \end{cases}
\end{multline}
holds, where $C$ is a constant independent of $n$ and $x$. Note that for $a,b\ge -1/2$ the previous bound can be replaced by the simpler one 
\begin{equation}
\label{eq:unif-bound}
|p_n^{(a,b)}(x)|\le C (1-x)^{-a/2-1/4}(1+x)^{-b/2-1/4}.
\end{equation}

\begin{propo}
\label{prop:well-def}
Let $\alpha,\beta\ge -1/2$, $\sigma>0$, and $f\in(\mathbb{C})^{\mathbb{N}}_0$. Then $(-\mathcal{J}^{(\alpha,\beta)})^{-\sigma}$ is well defined if and only if $\sigma<1/2$.
\end{propo}
\begin{proof}
First of all, we have that $W_t^{(\alpha,\beta)}f$ is well defined for $f\in \ell^{\infty}(\mathbb{N})$ (see~\cite{ACL-JacI}). Then, we will prove that $(-\mathcal{J}^{(\alpha,\beta)})^{-\sigma}$ is finite if and only if $0<\sigma<1/2$.

The sufficient argument is as follows. It is clear that
\begin{align*}
\left|(-\mathcal{J}^{(\alpha,\beta)})^{-\sigma}f(n)\right|&\le 
\frac{1}{\Gamma(\sigma)}\int_{0}^{\infty}|W_t^{(\alpha,\beta)}f(n)|\, \frac{dt}{t^{1-\sigma}}\\
&\le \frac{1}{\Gamma(\sigma)} \left(\int_{0}^{1} |W_{t}^{(\alpha,\beta)}f(n)|\frac{dt}{t^{1-\sigma}} + \int_{1}^{\infty} |W_{t}^{(\alpha,\beta)}f(n)|\frac{dt}{t^{1-\sigma}}\right)\\&:=\frac{I_1+I_2}{\Gamma(\sigma)}.
\end{align*}
For $I_1$ we use the estimate (see \cite[Lemma~3.2]{ACL-JacI} for the case $m\not=n$ and note that for $m=n$ is obvious)
\[
|K_t^{(\alpha,\beta)}(m,n)|\le C\begin{cases}t^{1/2}|m-n|^{-2}, & m\not=n,\\ 1, & m=n,\end{cases}
\]
to obtain that 
\begin{equation*}
I_1\le C \left(
\sum_{\begin{smallmatrix}
        m=0 \\
        m\not=n
      \end{smallmatrix}}^{\infty} \frac{|f(m)|}{|m-n|^{2}}\int_{0}^{1}\frac{dt}{t^{1/2-\sigma}} + |f(n)|\int_{0}^{1}\frac{dt}{t^{1-\sigma}}\right)
\end{equation*}
and both terms are finite for $\sigma>0$. To deduce the convergence $I_2$, using that $f\in (\mathbb{C})_0^{\mathbb{N}}$ and the bound \eqref{eq:unif-bound}, it is enough to show that
\begin{equation*}
\int_{1}^{\infty}\int_{-1}^{1} \frac{e^{-(1-x)t}}{\sqrt{1-x^{2}}}\,dx\frac{dt}{t^{1-\sigma}}<\infty.
\end{equation*}
Since
\[
\int_{-1}^{1}\frac
{e^{-(1-x)t}}{\sqrt{1-x^2}}\, dx\le C \int_{0}^{1}\frac
{e^{-(1-x)t}}{\sqrt{1-x}}\, dx = \frac{C}{\sqrt{t}}\int_{0}^{t}\frac{e^{-s}}{\sqrt{s}}\, ds\simeq \frac{C}{\sqrt{t}},
\]
we have
\[
\int_{1}^{\infty}\int_{-1}^{1} \frac{e^{-(1-x)t}}{\sqrt{1-x^{2}}}\,dx\frac{dt}{t^{1-\sigma}}\le C\int_{1}^{\infty}t^{\sigma-3/2}\, dt\le C,
\]
where we have used that $\sigma<1/2$.

To show the necessity of the condition $\sigma<1/2$, we will use the inequality
\[
\int_{-1}^{1}\frac{e^{-(1-x)t}}{\sqrt{1-x^2}}\, dx < \pi\liminf_{n\to \infty} \int_{-1}^{1}e^{-(1-x)t}(p_n^{(\alpha,\beta)}(x))^2\, d\mu_{\alpha,\beta}(x).
\]
This is a particular case of a classical result due to A. M\'at\'e, P. Nevai, and V. Totik, see \cite[Theorem~2]{MNT}. From this fact, there exists $N\in\mathbb{N}$ such that for every $n\geq N$,
\begin{equation*}
C \int_{-1}^{1}\frac{e^{-(1-x)t}}{\sqrt{1-x^2}}\, dx < \int_{-1}^{1}e^{-(1-x)t}(p_n^{(\alpha,\beta)}(x))^2\, d\mu_{\alpha,\beta}(x).
\end{equation*}
Then, taking $j\in\mathbb{N}$ such that $j\geq N$ and the sequence $\{f_j(m)=\delta_{jm}\}_{m\ge 0}$, where $\delta_{jm}$ stands for the Kronecker's delta, we have
\begin{equation}\label{eq:div-int}
\begin{aligned}
(-\mathcal{J}^{(\alpha,\beta)})^{-\sigma}f_j(j)&\ge
\int_{1}^{\infty} K_{t}^{(\alpha,\beta)}(j,j)\frac{dt}{t^{1-\sigma}} \\&= \int_{1}^{\infty}\int_{-1}^{1} e^{-(1-x)t}(p_{j}^{(\alpha,\beta)}(x))^{2}\,d\mu_{\alpha,\beta}(x)\frac{dt}{t^{1-\sigma}} \\&> C \int_{1}^{\infty} \int_{-1}^{1}\frac{e^{-(1-x)t}}{\sqrt{1-x^{2}}}\,dx \frac{dt}{t^{1-\sigma}}.
\end{aligned}
\end{equation}
Now, using that $t>1$, we obtain that
\[
\int_{-1}^{1}\frac{e^{-(1-x)t}}{\sqrt{1-x^2}}\, dx \ge \int_{0}^{1}\frac{e^{-(1-x)t}}{\sqrt{1-x}}\, dx = \frac{C}{\sqrt{t}}\int_{0}^{t}\frac{e^{-s}}{\sqrt{s}}\, ds\simeq \frac{C}{\sqrt{t}}.
\]
Then, since $(-\mathcal{J}^{(\alpha,\beta)})^{-\sigma}f_j(j)$ is well defined, from \eqref{eq:div-int} we deduce that $\sigma<1/2$.
\end{proof}

%%%%%%%%%%%%%%%%%%%%%%%%%%%%%%%%%%%%%%%%%%%%%%%%%%%%%%%%%%%%
\section{Proof of Theorem \ref{th:main}}
\label{sec:proof}
%%%%%%%%%%%%%%%%%%%%%%%%%%%%%%%%%%%%%%%%%%%%%%%%%%%%%%%%%%%%%%
We devote this section to prove Theorem \ref{th:main}. We will use the discrete Calder\'on-Zygmund theory so we first express the Riesz transform as in the form of Theorem~\ref{thm:CZ}.
From Proposition \ref{prop:well-def}, for $\alpha,\beta\geq-1/2$, $0<\sigma<1/2$, and $f\in(\mathbb{C})_{0}^{\mathbb{N}}$, applying Fubini's theorem we obtain that
\begin{multline*}
(-\mathcal{J}^{(\alpha,\beta)})^{-\sigma}f(n)=\sum_{m=0}^{\infty}f(m)\frac{1}{\Gamma(\sigma)}\int_{0}^{\infty}K_t^{(\alpha,\beta)}(m,n)\, \frac{dt}{t^{1-\sigma}}\\
\begin{aligned}
&=\frac{1}{\Gamma(\sigma)}\sum_{m=0}^{\infty}f(m)\int_{-1}^{1}p_m^{(\alpha,\beta)}(x)p_n^{(\alpha,\beta)}(x)\int_{0}^{\infty}t^{\sigma-1}e^{-(1-x)t}\, dt\, d\mu_{\alpha,\beta}(x)\\&=\sum_{m=0}^{\infty}f(m)\int_{-1}^{1}\frac{p_m^{(\alpha,\beta)}(x)p_n^{(\alpha,\beta)}(x)}{(1-x)^{\sigma}}\, d\mu_{\alpha,\beta}(x).
\end{aligned}
\end{multline*}
By \cite[18.9.6]{NIST}, it is easy to check that
\[
\delta p_n^{(\alpha,\beta)}(x)=(1-x)p_n^{(\alpha+1,\beta)}(x),
\]
and therefore, for each sequence in $f\in (\mathbb{C})_0^{\mathbb{N}}$,
\begin{align}
\label{eq:def-R}
\mathcal{R}f(n)&=\lim_{\sigma\to \frac{1}{2}^{-}}\delta (-\mathcal{J}^{(\alpha,\beta)})^{-\sigma}f(n)\notag\\&=\lim_{\sigma \to \frac{1}{2}^{-}}\sum_{m=0}^{\infty} f(m)\int_{-1}^{1}\frac{p_m^{(\alpha,\beta)}(x)p_n^{(\alpha+1,\beta)}(x)}{(1-x)^{\sigma-1}}\, d\mu_{\alpha,\beta}(x)\notag\\&=
\sum_{m=0}^{\infty} f(m)R(m,n),
\end{align}
with
\[
R(m,n)=\int_{-1}^{1}(1-x)^{1/2}p_m^{(\alpha,\beta)}(x)p_n^{(\alpha+1,\beta)}(x)\, d\mu_{\alpha,\beta}(x).
\]

Now, the following propositions allow us to obtain conditions (a) and (b) for some kernels that will be defined later.

\begin{propo} \label{propo:Riesz-size}
  Let $n,m\in\mathbb{N}$, $n\neq m$, $\alpha,\beta\ge  -1/2$. Then
  \begin{equation}
  \label{eq:Riesz-size}
  |R(m,n)|\le \frac{C}{|m-n|}.
  \end{equation}
\end{propo}

\begin{propo}\label{propo:Riesz-smooth}
  Let $n,m\in\mathbb{N}$, $n\neq m$, $m/2\leq n\leq 3m/2$, $\alpha,\beta\ge -1/2$. Then
  \begin{equation}
  \label{eq:Riesz-smooth-1}
  |R(m+2,n)-R(m,n)|\le \frac{C}{|m-n|^{2}}
  \end{equation}
and
 \begin{equation}
 \label{eq:Riesz-smooth-2}
  |R(m,n+2)-R(m,n)|\le \frac{C}{|m-n|^{2}}.
  \end{equation}
\end{propo}

The proofs of the previous propositions are the most delicate points of the paper and they are postponed to the next section.

Finally, we state the next lemma concerning $A_p(\mathbb{N})$ weights, see \cite[Lemma 2.2]{ACL-Trans}, before giving the proof of Theorem \ref{th:main}.
\begin{lem}
\label{lem:weight}
  Let $1\le p <\infty$ and $w\in A_p(\mathbb{N})$. Then, $w(n)\simeq w(n+1)$.
\end{lem}

\begin{proof}[Proof of Theorem \ref{th:main}]
First, we will see that $\mathcal{R}$ is bounded operator from $\ell^2(\mathbb{N})$ into itself. To this end, by denseness, it is enough to consider sequences in $(\mathbb{C})_0^{\mathbb{N}}$, so \eqref{eq:def-R} can be used.

As it is well known, for each function $f\in L^2([-1,1],d\mu_{\alpha,\beta})$ its Fourier-Jacobi coefficients are given by
\[
c_m^{(\alpha,\beta)}(f)=\int_{-1}^{1}f(x)p_m^{(\alpha,\beta)}(x)\,d\mu_{\alpha,\beta}(x)
\]
and
\[
f(x)=\sum_{m=0}^{\infty}c_m^{(\alpha,\beta)}(f)p_m^{(\alpha,\beta)}(x),
\]
where the equality holds in $ L^2([-1,1],d\mu_{\alpha,\beta})$. Moreover, $\{c_m^{(\alpha,\beta)}(f)\}_{m\ge 0}$ is a sequence in $\ell^{2}(\mathbb{N})$. Conversely, for each sequence $f\in \ell^2(\mathbb{N})$, the function
\begin{equation}
\label{eq:Jac-transform}
F(x)=\sum_{m=0}^{\infty}f(m)p_m^{(\alpha,\beta)}(x)
\end{equation}
belongs to $L^2([-1,1],d\mu_{\alpha,\beta})$ and Parseval's identity
\begin{equation}
\label{eq:parse}
\|f\|_{\ell^2(\mathbb{N})}=\|F\|_{L^2([-1,1],d\mu_{\alpha,\beta})}
\end{equation}
holds. 
Therefore, noting that
\begin{align*}
\mathcal{R}f(n)¨&=\int_{-1}^{1}(1-x)^{1/2}p_n^{(\alpha+1,\beta)}(x)F(x)\, d\mu_{\alpha,\beta}(x)\\&=c_n^{(\alpha+1,\beta)}((1-\cdot)^{-1/2}F),
\end{align*}
where $F$ is defined as in \eqref{eq:Jac-transform}, by \eqref{eq:parse} we have
\begin{align*}
\|\mathcal{R}f\|_{\ell^{2}(\mathbb{N})}&=\|c_n^{(\alpha+1,\beta)}((1-\cdot)^{-1/2}F)\|_{\ell^{2}(\mathbb{N})}\\
&=\|(1-\cdot)^{-1/2}F\|_{L^2([-1,1],d\mu_{\alpha+1,\beta})}=\|F\|_{L^2([-1,1],d\mu_{\alpha,\beta})}=\|f\|_{\ell^2(\mathbb{N})}
\end{align*}
and then $\mathcal{R}$ is a bounded operator from $\ell^2(\mathbb{N})$ into itself.

Now, we note that it is possible to split the $m$ variable into its even and odd parts, so we have
\begin{equation*}
\mathcal{R} f(n) = \sum_{m=0}^{\infty} f(2m) R(2m,n) + \sum_{m=0}^{\infty} f(2m+1) R(2m+1,n),
\end{equation*}
which motivates the following definitions
\begin{align*}
{}^{\text{e,e}}\mathcal{R}f(n) &= \sum_{m=0}^{\infty} f(m) {}^{\text{e,e}}R(m,n), & {}^{\text{e,e}}R(m,n) &= R(2m,2n),\\
{}^{\text{e,o}}\mathcal{R}f(n) &= \sum_{m=0}^{\infty} f(m) {}^{\text{e,o}}R(m,n), & {}^{\text{e,o}}R(m,n) &= R(2m+1,2n),\\
{}^{\text{o,e}}\mathcal{R}f(n) &= \sum_{m=0}^{\infty} f(m) {}^{\text{o,e}}R(m,n), & {}^{\text{o,e}}R(m,n) &= R(2m,2n+1),
\intertext{and}
{}^{\text{o,o}}\mathcal{R}f(n) &= \sum_{m=0}^{\infty} f(m) {}^{\text{o,o}}R(m,n), & {}^{\text{o,o}}R(m,n) &= R(2m+1,2n+1).
\end{align*}
Hence, we obtain that
\begin{equation*}
\mathcal{R}f(2n) = {}^{\text{e,e}}\mathcal{R} \tilde{f}(n) + {}^{\text{e,o}}\mathcal{R} \hat{f}(n),
\end{equation*}
and
\begin{equation*}
\mathcal{R} f(2n+1) = {}^{\text{o,e}}\mathcal{R} \tilde{f}(n) + {}^{\text{o,o}}\mathcal{R} \hat{f}(n),
\end{equation*}
with $\tilde{f}(n) = f(2n)$ and $\hat{f}(n) = f(2n+1)$, $n\in\mathbb{N}$. In addition, note that $^{\text{e,e}}\mathcal{R}$, $^{\text{e,o}}\mathcal{R}$, $^{\text{o,e}}\mathcal{R}$, and $^{\text{o,o}}\mathcal{R}$ are bounded operators in $\ell^{2}(\mathbb{N})$ because so is $\mathcal{R}$. Indeed, let us define the functions
\begin{equation*}
g(n) = f(n/2)\chi_{\mathcal{E}}(n)
\qquad \text{ and } \qquad
h(n) = f((n-1)/2)\chi_{\mathcal{O}}(n),
\end{equation*}
where $\mathcal{E}$ and $\mathcal{O}$ denotes the sets of even and odd numbers respectively. We have then that  ${}^{\text{e,e}}\mathcal{R} f(n) = \mathcal{R} g(2n)$, ${}^{\text{e,o}}\mathcal{R} f(n) = \mathcal{R} h(2n)$, ${}^{\text{o,e}}\mathcal{R} f(n) = \mathcal{R} g(2n+1)$, and ${}^{\text{o,o}}\mathcal{R} f(n) = \mathcal{R} h(2n+1)$, so the boundedness in $\ell^2(\mathbb{N})$ of each operator follows immediately.

Therefore, it is enough to prove that the kernels ${}^{\text{e,e}}R$, ${}^{\text{e,o}}R$, ${}^{\text{o,e}}R$, and ${}^{\text{o,o}}R$ satisfy  properties (a) and (b). These facts are immediate consequences of Propositions \ref{propo:Riesz-size} and \ref{propo:Riesz-smooth}.

In this way, by Theorem \ref{thm:CZ} and taking the weights $w_e(n)=w(2n)$ and $w_o(n)=w(2n+1)$ (note that both of them belongs to $A_p(\mathbb{N})$ because $w\in A_p(\mathbb{N})$), for $1<p<\infty$ we have
\[
\|{}^{\text{e,e}}\mathcal{R} \tilde{f}\|_{\ell^p(\mathbb{N},w_e)}\le \|\tilde{f}\|_{\ell^p(\mathbb{N},w_e)},
\]
\[
\|{}^{\text{e,o}}\mathcal{R} \hat{f}\|_{\ell^p(\mathbb{N},w_e)}\le \|\hat{f}\|_{\ell^p(\mathbb{N},w_e)},
\]
\[
\|{}^{\text{o,e}}\mathcal{R} \tilde{f}\|_{\ell^p(\mathbb{N},w_o)}\le \|\tilde{f}\|_{\ell^p(\mathbb{N},w_o)},
\]
\[
\|{}^{\text{o,o}}\mathcal{R} \hat{f}\|_{\ell^p(\mathbb{N},w_o)}\le \|\hat{f}\|_{\ell^p(\mathbb{N},w_o)},
\]
and the corresponding weak inequalities for $p=1$. To complete the proof, it is enough to observe that, by Lemma \ref{lem:weight},
\[
\|\hat{f}\|_{\ell^p(\mathbb{N},w_e)}\le C \|\hat{f}\|_{\ell^p(\mathbb{N},w_o)}\le C\|f\|_{\ell^p(\mathbb{N},w)}
\]
and
\[
\|\tilde{f}\|_{\ell^p(\mathbb{N},w_o)}\le C \|\hat{f}\|_{\ell^p(\mathbb{N},w_e)}\le C\|f\|_{\ell^p(\mathbb{N},w)}.\qedhere
\]
\end{proof}

%%%%%%%%%%%%%%%%%%%%%%%%%%%%%%%%%%%%%%%%%%%%%%%%%%%%%%%%%%%%%%%
\section{Proof of Propositions \ref{propo:Riesz-size} and \ref{propo:Riesz-smooth}}
\label{sec:estimates}
%%%%%%%%%%%%%%%%%%%%%%%%%%%%%%%%%%%%%%%%%%%%%%%%%%%%%%%%%%%%%%%

\begin{proof}[Proof of Proposition \ref{propo:Riesz-size}]
First we note that
\[
L^{a,b}p_n^{(a,b)}(x)=\lambda_n^{(a,b)}p_n^{(a,b)}(x),
\]
with $\lambda_{n}^{(a,b)}=n(n+a+b+1)$ and
\[
L^{a,b}=-(1-x^2)\frac{d^2}{dx^2}-(b-a-(a+b+2)x)\frac{d}{dx}.
\]
It is well known that $L^{a,b}$ is a symmetric operator in $L^2([-1,1],d\mu_{a,b})$, but for some interval $[r,s]\subset [-1,1]$, $r<s$, it is verified that
\begin{equation}
\label{eq:L-parts}
\int_{r}^{s}f(x)L^{a,b}g(x)\, d\mu_{a,b}(x)=U_{a,b}(f,g)(x)\Big|_{x=r}^{x=s}+\int_{r}^{s}g(x)L^{a,b}f(x)\, d\mu_{a,b}(x),
\end{equation}
with
\[
U_{a,b}(f,g)(x)=(1-x)^{a+1}(1+x)^{b+1}\Big(g(x)\frac{df}{dx}(x)-f(x)\frac{dg}{dx}(x)\Big).
\]
Moreover,
\begin{multline}
\label{eq:L-product}
L^{a,b}(h_1h_2)(x)=h_2(x)L^{a+1,b}h_1(x)-(1+x)h_2(x)\frac{dh_1(x)}{dx}-2(1-x^2)\frac{dh_1(x)}{dx}\frac{dh_2(x)}{dx}\\-(1-x^2)h_1(x)\frac{d^2h_2(x)}{dx^2}-
(b-a-(a+b+2)x)h_1(x)\frac{dh_2(x)}{dx}
\end{multline}
and
\begin{multline}
\label{eq:L-product-2}
L^{a+1,b}(h_1h_2)(x)=h_2(x)L^{a,b}h_1(x)+(1+x)h_2(x)\frac{dh_1}{dx}(x)-2(1-x^2)\frac{dh_1}{dx}(x)\frac{dh_2}{dx}(x)\\-(1-x^2)h_1(x)\frac{d^2h_2}{dx^2}(x)-(b-a-1-(a+b+3)x)h_1(x)\frac{dh_2}{dx}(x).
\end{multline}

First, we suppose that $n>m$. We decompose $R(m,n)$ according to the intervals $I_1=(-1,-1+1/(n+1)^2)$, $I_2=[-1+1/(n+1)^2, 1-1/(n+1)^2]$, and $I_3=(1-1/(n+1)^2,1)$ and denote the corresponding integrals by $R_1(m,n)$, $R_2(m,n)$, and $R_3(m,n)$. From \eqref{eq:unif-bound-trozos}, for $\alpha,\beta\ge -1/2$, we have
\[
|R_1(m,n)|\le C (n+1)^{\beta+1/2}(m+1)^{\beta+1/2}\int_{I_1} (1+x)^\beta\, dx\le \frac{C}{n+1}
\]
and
\[
|R_3(m,n)|\le C (n+1)^{\alpha+3/2}(m+1)^{\alpha+1/2}\int_{I_3}(1-x)^{\alpha+1/2}\, dx\le \frac{C}{n+1},
\]
and these estimates are enough to prove \eqref{eq:Riesz-size}.

Let us focus on $R_2(m,n)$. We consider the notation
\[
J(m,n)=\int_{I_2}H_{\alpha,\beta}(x)p_n^{(\alpha+1,\beta)}(x)p_m^{(\alpha,\beta)}(x)\,d\mu_{\alpha,\beta}(x),
\]
with
\begin{equation}
\label{eq:H.function}
H_{\alpha,\beta}(x)=\frac{2\beta-2\alpha+1-(2\alpha+2\beta+3)x}{4(1-x)^{1/2}},
\end{equation}
and
\[
S(m,n)=U_{\alpha,\beta}((1-(\cdot))^{1/2}p_n^{(\alpha+1,\beta)},p_m^{(\alpha,\beta)})(x)\Big|_{x=-1+1/(n+1)^2}^{x=1-1/(n+1)^2}.
\]
To give a proper expression for the integral $R_2(m,n)$, we use \eqref{eq:L-parts}, with $f(x)=(1-x)^{1/2}p_n^{(\alpha+1,\beta)}(x)$ and $g(x)=p_m^{(\alpha,\beta)}(x)$, and \eqref{eq:L-product}, with $h_1(x)=p_n^{(\alpha+1,\beta)}(x)$ and $h_2(x)=(1-x)^{1/2}$. Then, we get that
\begin{align*}
\lambda_m^{(\alpha,\beta)}R_2(m,n)&=\int_{I_2}(1-x)^{1/2}p_n^{(\alpha+1,\beta)}(x)L^{\alpha,\beta}p_m^{(\alpha,\beta)}(x)\,d\mu_{\alpha,\beta}(x)
\\&=S(m,n)+\int_{I_2}L^{\alpha,\beta}((1-(\cdot))^{1/2}p_n^{(\alpha+1,\beta)})(x)p_m^{(\alpha,\beta)}(x)\,d\mu_{\alpha,\beta}(x)
\\&=S(m,n)+\lambda_n^{(\alpha+1,\beta)}R_{2}(m,n)+J(m,n).
\end{align*}
Therefore, noting that $\lambda_m^{(\alpha,\beta)}\not= \lambda_n^{(\alpha+1,\beta)}$,
\begin{equation}
\label{eq:R-size}
R_{2}(m,n)=\frac{S(m,n)+J(m,n)}{\lambda_m^{(\alpha,\beta)}-\lambda_n^{(\alpha+1,\beta)}}.
\end{equation}
Now, we use the identities (see~\cite[18.9.15]{NIST})
\begin{equation}
\label{eq:Jaco-der}
\frac{d P_n^{(a,b)}}{dx}(x)=\frac{n+a+b+1}{2}P_{n-1}^{(a+1,b+1)}(x),\qquad n> 0,
\end{equation}
and
\begin{equation}
\label{eq:Jaco-der0}
\frac{d P_0^{(a,b)}}{dx}(x)=0,
\end{equation}
the estimate \eqref{eq:unif-bound-trozos}, and the restrictions $\alpha,\beta\geq -1/2$ to obtain that
\begin{equation}
\label{eq:S}
|S(m,n)|\le C (n+1).
\end{equation}
In order to estimate the term $J(m,n)$ we decompose it according to the intervals $V_1=[-1+1/(n+1)^2,-1+1/(m+1)^2)$, $V_2=[-1+1/(m+1)^2,1-1/(m+1)^2]$, and $V_3=(1-1/(m+1)^2,1-1/(n+1)^2]$. We denote the corresponding integrals by $J_1(m,n)$, $J_2(m,n)$, and $J_3(m,n)$. In this way, using \eqref{eq:unif-bound-trozos}, the estimate $|H_{\alpha,\beta}(x)|\le C (1-x)^{-1/2}$ for $-1<x<1$, and the condition $\alpha,\beta\ge -1/2$, we deduce the bounds
\begin{align*}
|J_1(m,n)|&\le C (m+1)^{\beta+1/2}\int_{V_1}(1+x)^{\beta/2-1/4}\, dx\\&\le C\int_{V_1}(1+x)^{-1/2}\,dx\le C,
\end{align*}
\begin{align*}
|J_2(m,n)|&\le C \int_{V_2}(1+x)^{-1/2}(1-x)^{-3/2}\, dx\\&\le C (m+1),
\end{align*}
and 
\begin{align*}
|J_3(m,n)|&\le C (m+1)^{\alpha+1/2}\int_{V_3}(1-x)^{\alpha/2-5/4}\, dx\\&\le C \int_{V_3}(1-x)^{-3/2}\,dx\le C (n+1).
\end{align*}
Then, we have
\begin{equation}
\label{eq:J}
|J(m,n)|\le C (n+1),
\end{equation}
and, from \eqref{eq:R-size}, \eqref{eq:S}, and \eqref{eq:J}, we obtain that $|R_2(m,n)|\le C|n-m|^{-1}$ and the estimate \eqref{eq:Riesz-size} is proved for $n>m$.

The case $n<m$ follows from the above argument by interchanging the roles of $n$ and $m$ but we include some details for the sake of completeness.

We decompose $R(m,n)$ according to the intervals $I'_1=(-1,-1+1/(m+1)^2)$, $I'_2=[-1+1/(m+1)^2, 1-1/(m+1)^2]$, and $I'_3=(1-1/(m+1)^2,1)$ and denote the corresponding integrals by $R'_1(m,n)$, $R'_2(m,n)$, and $R'_3(m,n)$. By similar arguments than above we obtain that
\[
|R'_1(m,n)|\le \frac{C}{m+1}\qquad \text{ and }\qquad |R'_3(m,n)|\le \frac{C}{m+1}.
\]
Now, for $R_2'(m,n)$, by using \eqref{eq:L-product-2} and noting again that $\lambda_m^{(\alpha,\beta)}\not= \lambda_n^{(\alpha+1,\beta)}$, we deduce the identity
\begin{equation}
\label{eq:R'-size}
R'_{2}(m,n)=\frac{S'(m,n)-J'(m,n)}{\lambda_n^{(\alpha+1,\beta)}-\lambda_m^{(\alpha,\beta)}},
\end{equation}
where
\[
J'(m,n)=\int_{I'_2}H_{\alpha,\beta}(x)p_n^{(\alpha+1,\beta)}(x)p_m^{(\alpha,\beta)}(x)\,d\mu_{\alpha,\beta}(x),
\]
with $H_{\alpha,\beta}$ as in \eqref{eq:H.function},
and
\[
S'(m,n)=U_{\alpha+1,\beta}((1-(\cdot))^{-1/2}p_m^{(\alpha,\beta)},p_n^{(\alpha+1,\beta)})(x)\Big|_{x=-1+1/(m+1)^2}^{x=1-1/(m+1)^2}.
\]
As in the previous case, we deduce the estimate
\begin{equation}
\label{eq:Sprim}
|S'(m,n)|\leq (m+1).
\end{equation}
To analyze $J'(m,n)$ we decompose it according to the intervals $V'_1=[-1+1/(m+1)^2,-1+1/(n+1)^2)$, $V'_2=[-1+1/(n+1)^2,1-1/(n+1)^2]$, and $V'_3=(1-1/(n+1)^2,1-1/(m+1)^2]$. The corresponding integrals are denoted by $J_{1}'(m,n)$, $J_{2}'(m,n)$, and $J_{3}'(m,n)$, and we have
\begin{equation*}
|J_{1}'(m,n)|\leq C,\qquad |J_{2}'(m,n)|\le C(n+1),\qquad \text{ and }\qquad |J_{3}'(m,n)|\leq C (m+1).
\end{equation*}
Therefore
\begin{equation}
\label{eq:J'}
|J'(m,n)|\leq C(m+1).
\end{equation}
Then \eqref{eq:Riesz-size} is also proved for $n<m$ and the proof of the proposition is finished.
\end{proof}

In the proof of the Proposition \ref{propo:Riesz-smooth} we will use the following lemmas.

\begin{lem}
\label{lem:bound-diff}
  Let $n\in\mathbb{N}$ and $a,b> -1$, then
\begin{multline*}
%\label{eq:bound-diff}
|p_{n+2}^{(a,b)}(x)-p_{n}^{(a,b)}(x)|\\\le C
\begin{cases}
(n+1)^{a-1/2}, & 1-1/(n+1)^2<x<1,\\
(1-x)^{-a/2+1/4}(1+x)^{-b/2+1/4}, & -1+1/(n+1)^2\le x\leq1-1/(n+1)^{2},\\
(n+1)^{b-1/2}, & -1<x<-1+1/(n+1)^2.
\end{cases}
\end{multline*}
\end{lem}

\begin{lem}
\label{lem:bound-diff-der}
  Let $n\in\mathbb{N}$, $a,b> -1$, then
\begin{multline*}
%\label{eq:bound-diff-der}
|(p_{n+2}^{(a,b)}-p_{n}^{(a,b)})'(x)|\le C (n+1)\\ \times
\begin{cases}
(n+1)^{a+1/2}, & 1-1/(n+1)^2<x<1,\\
(1-x)^{-a/2-1/4}(1+x)^{-b/2-1/4}, & -1+1/(n+1)^2\le x \leq 1-1/(n+1)^{2},\\
(n+1)^{b+1/2}, & -1<x<-1+1/(n+1)^2.
\end{cases}
\end{multline*}
\end{lem}
We postpone the proof of these two lemmas to the last section of the paper.

\begin{proof}[Proof of Proposition \ref{propo:Riesz-smooth}]
We will prove the estimate \eqref{eq:Riesz-smooth-1} for $n>m$ and \eqref{eq:Riesz-smooth-2} for $n<m$. The remaining two cases can be treated in a similar way and we omit the details.

In this way, we first assume that $n>m$ and prove \eqref{eq:Riesz-smooth-1}.

We decompose the difference $R(m+2,n)-R(m,n)$ into three integrals $\mathcal{R}_1(m,n)$, $\mathcal{R}_2(m,n)$, and $\mathcal{R}_3(m,n)$ over the intervals $I_1=(-1,-1+1/(n+1)^2)$, $I_2=[-1+1/(n+1)^2, 1-1/(n+1)^2]$, and $I_3=(1-1/(n+1)^2,1)$. From \eqref{eq:unif-bound-trozos} and Lemma \ref{lem:bound-diff} (note that by hypothesis $m/2\leq n\leq 3m/2$), we have
\[
|\mathcal{R}_1(m,n)|\le C (m+1)^{\beta-1/2}(n+1)^{\beta+1/2}\int_{I_1} (1+x)^\beta\, dx\le \frac{C}{(n+1)^{2}}
\]
and
\[
|\mathcal{R}_3(m,n)|\le C (n+1)^{\alpha+3/2}(m+1)^{\alpha-1/2}\int_{I_3}(1-x)^{\alpha+1/2}\, dx\le \frac{C}{(n+1)^{2}},
\]
which are enough to prove \eqref{eq:Riesz-smooth-1}.

We deal now with the most delicate integral $\mathcal{R}_2(m,n)$.  We recover some notation from the proof of Proposition~\ref{propo:Riesz-size} and denote
\begin{equation*}
\mathcal{J}(m,n)=\int_{I_2}H_{\alpha,\beta}(x)p_n^{(\alpha+1,\beta)}(x)
(p_{m+2}^{(\alpha,\beta)}(x)-p_m^{(\alpha,\beta)}(x))\,d\mu_{\alpha,\beta}(x),
\end{equation*}
and
\[
\mathcal{S}(m,n)=U_{\alpha,\beta}((1-(\cdot))^{1/2}p_n^{(\alpha+1,\beta)},p_{m+2}^{(\alpha,\beta)}-p_m^{(\alpha,\beta)})(x)
\Big|_{x=-1+1/(n+1)^2}^{x=1-1/(n+1)^2}.
\]
By \eqref{eq:R-size}, using that $\lambda_{m+2}^{(\alpha,\beta)}\not= \lambda_n^{(\alpha+1,\beta)}$ and $\lambda_{m}^{(\alpha,\beta)}\not= \lambda_n^{(\alpha+1,\beta)}$, we obtain that
\begin{multline}
\label{eq:R-smooth}
\mathcal{R}_2(m,n)=\frac{S(m+2,n)+J(m+2,n)}{\lambda_{m+2}^{(\alpha,\beta)}-\lambda_n^{(\alpha+1,\beta)}}
-\frac{S(m,n)+J(m,n)}{\lambda_{m}^{(\alpha,\beta)}-\lambda_n^{(\alpha+1,\beta)}}\\
=\frac{\mathcal{S}(m,n)+\mathcal{J}(m,n)}{\lambda_{m+2}^{(\alpha,\beta)}-\lambda_n^{(\alpha+1,\beta)}}
-\frac{2(2m+\alpha+\beta+3)(S(m,n)+J(m,n))}
{(\lambda_{m+2}^{(\alpha,\beta)}-\lambda_n^{(\alpha+1,\beta)})(\lambda_{m}^{(\alpha,\beta)}-\lambda_n^{(\alpha+1,\beta)})}.
\end{multline}
We use \eqref{eq:S} and \eqref{eq:J} to obtain that
\begin{equation}
\label{eq:smooth-aux-1}
\left|\frac{2(2m+\alpha+\beta+3)(S(m,n)+J(m,n))}
{(\lambda_{m+2}^{(\alpha,\beta)}-\lambda_n^{(\alpha+1,\beta)})(\lambda_{m}^{(\alpha,\beta)}-\lambda_n^{(\alpha+1,\beta)})}\right|\le \frac{C}{|n-m|^2}.
\end{equation}

From \eqref{eq:unif-bound-trozos}, \eqref{eq:Jaco-der}, Lemmas \ref{lem:bound-diff} and \ref{lem:bound-diff-der}, we have 
\[
|\mathcal{S}(m,n)|\le C
\]
and hence
\begin{equation}
\label{eq:smooth-aux-2}
\left|\frac{\mathcal{S}(m,n)}{\lambda_{m+2}^{(\alpha,\beta)}-\lambda_n^{(\alpha+1,\beta)}}\right|\le \frac{C}{|n-m|^2}.
\end{equation}

Now, to analyse the term $\mathcal{J}(m,n)$ we will use \eqref{eq:L-product-2}. 
Therefore, taking the notation
\[
\overline{\mathcal{S}}(m,n)=U_{\alpha+1,\beta}\Big(\mathcal{H}_{\alpha,\beta}(p_{m+2}^{(\alpha,\beta)}-p_m^{(\alpha,\beta)}),
p_n^{(\alpha+1,\beta)}\Big)(x)\Big|_{x=-1+1/(n+1)^2}^{x=1-1/(n+1)^2},
\]
where
\[
\mathcal{H}_{\alpha,\beta}(x)=\frac{H_{\alpha,\beta}(x)}{1-x},
\]
\begin{multline*}
T_1(m,n)=\int_{I_2}((1+x)\mathcal{H}_{\alpha,\beta}(x)-2(1-x^2)\mathcal{H}'_{\alpha,\beta}(x))\\(p_{m+2}^{(\alpha,\beta)}-p_m^{(\alpha,\beta)})'(x)p_n^{(\alpha+1,\beta)}(x)\, d\mu_{\alpha+1,\beta}(x),
\end{multline*}
and
\begin{multline*}
T_2(m,n)=\int_{I_2}((1-x^2)\mathcal{H}''_{\alpha,\beta}(x)+(\beta-\alpha-1-(\alpha+\beta+3)x)\mathcal{H}'_{\alpha,\beta}(x))
\\(p_{m+2}^{(\alpha,\beta)}(x)-p_m^{(\alpha,\beta)}(x))
p_n^{(\alpha+1,\beta)}(x)\, d\mu_{\alpha+1,\beta}(x),
\end{multline*}
we have
\begin{multline*}
\lambda_{n}^{(\alpha+1,\beta)}\mathcal{J}(m,n)\\
\begin{aligned}
&=
\int_{I_2}\mathcal{H}_{\alpha,\beta}(x)(p_{m+2}^{(\alpha,\beta)}(x)-p_m^{(\alpha,\beta)}(x))
L^{\alpha+1,\beta}p_n^{(\alpha+1,\beta)}(x)\,d\mu_{\alpha+1,\beta}(x)\\
&=\overline{\mathcal{S}}(m,n)+\int_{I_2}L^{\alpha+1,\beta}(\mathcal{H}_{\alpha,\beta}(p_{m+2}^{(\alpha,\beta)}
-p_m^{(\alpha,\beta)}))(x)p_n^{(\alpha+1,\beta)}(x)\,d\mu_{\alpha+1,\beta}(x)\\
&=\overline{\mathcal{S}}(m,n)+\int_{I_2}H_{\alpha,\beta}(x)L^{\alpha,\beta}(p_{m+2}^{(\alpha,\beta)}-p_m^{(\alpha,\beta)}))(x)
p_n^{(\alpha+1,\beta)}(x)\,d\mu_{\alpha,\beta}(x)\\&\kern25pt +T_1(m,n)-T_2(m,n).
\end{aligned}
\end{multline*}
We use now the identity
\begin{equation}\label{eq_Lpols}
L^{\alpha,\beta}(p_{m+2}^{(\alpha,\beta)}-p_m^{(\alpha,\beta)})(x)
=\lambda_{m+2}^{(\alpha,\beta)}(p_{m+2}^{(\alpha,\beta)}(x)-p_m^{(\alpha,\beta)}(x))+(\lambda_{m+2}^{(\alpha,\beta)}
-\lambda_m^{(\alpha,\beta)})p_m^{(\alpha,\beta)}(x)
\end{equation}
to deduce that
\begin{multline*}
\lambda_{n}^{(\alpha+1,\beta)}\mathcal{J}(m,n)=\overline{\mathcal{S}}(m,n)+\lambda_{m+2}^{(\alpha,\beta)}\mathcal{J}(m,n)\\+2(2m+\alpha+\beta+3)J(m,n) +T_1(m,n)-T_2(m,n)
\end{multline*}
In this way,
\begin{equation}
\label{eq:smooth-aux-3}
\frac{\mathcal{J}(m,n)}{\lambda_{m+2}^{(\alpha,\beta)}-\lambda_n^{(\alpha+1,\beta)}}=
\frac{-\overline{\mathcal{S}}(m,n)-2(2m+\alpha+\beta+3)J(m,n)-T_1(m,n)+T_2(m,n)}{(\lambda_{m+2}^{(\alpha,\beta)}-\lambda_n^{(\alpha+1,\beta)})^2}.
\end{equation}
From \eqref{eq:J}, we deduce the estimate
\[
\left|\frac{2(2m+\alpha+\beta+3)J(m,n)}{(\lambda_{m+2}^{(\alpha,\beta)}-\lambda_n^{(\alpha+1,\beta)})^2}\right|\le \frac{C}{|n-m|^2}.
\]
Then, it suffices to show that
\begin{equation}
\label{eq:aux-prop}
|\overline{\mathcal{S}}(m,n)|+|T_1(m,n)|+|T_2(m,n)|\le C (n+1)^2
\end{equation}
because using \eqref{eq:R-smooth}, \eqref{eq:smooth-aux-1}, \eqref{eq:smooth-aux-2}, and \eqref{eq:smooth-aux-3}, the proof of \eqref{eq:Riesz-smooth-1} for $n>m$ will be completed.

From  \eqref{eq:unif-bound-trozos}, \eqref{eq:Jaco-der}, Lemmas \ref{lem:bound-diff} and \ref{lem:bound-diff-der}, and using the bounds $|\mathcal{H}_{\alpha,\beta}(x)|\leq C(1-x)^{-3/2}$ and $|\mathcal{H}_{\alpha,\beta}'(x)|\leq C(1-x)^{-5/2}$, for $-1<x<1$, we obtain the estimate
\begin{equation*}
|\overline{\mathcal{S}}(m,n)|\le C (n+1)^2.
\end{equation*}

Now we decompose $T_1(m,n)$ and $T_2(m,n)$ according the intervals $V_1=[-1+1/(n+1)^2,-1+1/(m+1)^2)$, $V_2=[-1+1/(m+1)^2,1-1/(m+1)^2]$, and $V_3=(1-1/(m+1)^2,1-1/(n+1)^2]$. Using \eqref{eq:unif-bound-trozos}, Lemma \ref{lem:bound-diff-der}, and the estimate
\[
|(1+x)\mathcal{H}_{\alpha,\beta}(x)-2(1-x^2)\mathcal{H}'_{\alpha,\beta}(x)|\le C (1+x)(1-x)^{-3/2}, \qquad -1<x<1,
\]
for $m/2\leq n\leq 3m/2$ and $\alpha\ge -1/2$ we have
\begin{align*}
|T_1(m,n)|\le &C\left((m+1)^{\beta+3/2}\int_{V_1}(1+x)^{\beta/2+3/4}\, dx\right.\\&
\kern 20pt +(m+1)\int_{V_2}(1+x)^{1/2}(1-x)^{-3/2}\, dx\\&\left.\kern20pt+(m+1)^{\alpha+3/2}\int_{V_3}(1-x)^{\alpha/2-5/4}\, dx\right)\le C(n+1)^2.
\end{align*}
Finally,  by \eqref{eq:unif-bound-trozos}, Lemma \ref{lem:bound-diff}, and the bound
\[
|(1-x^2)\mathcal{H}''_{\alpha,\beta}(x)+2(\beta-\alpha-1-(\alpha+\beta+3)x)\mathcal{H}'_{\alpha,\beta}(x)|\le C (1-x)^{-5/2}, \quad -1<x<1,
\]
we can show that for $m/2\leq n\leq 3m/2$ and $\alpha\ge -1/2$,
\begin{multline*}
|T_2(m,n)|\le C\left((m+1)^{\beta-1/2}\int_{V_1}(1+x)^{\beta/2-1/4}\, dx+\int_{V_2}(1-x)^{-2}\, dx\right.\\\left.+(m+1)^{\alpha-1/2}\int_{V_3}(1-x)^{\alpha/2-9/4}\, dx\right)\le C(n+1)^2,
\end{multline*}
and the proof of \eqref{eq:aux-prop} is completed.

Now we will prove the estimate \eqref{eq:Riesz-smooth-2} for $n<m$.

Again, we decompose the difference $R(m,n+2)-R(m,n)$ into three integrals $\mathcal{R}_{1}'(m,n)$, $\mathcal{R}_{2}'(m,n)$, and $\mathcal{R}_{3}'(m,n)$, over the intervals $I_{1}'=(-1,-1+1/(m+1)^{2})$, $I_{2}'=[-1+1/(m+1)^{2},1-1/(m+1)^{2}]$, $I_{3}'=(1-1/(m+1)^{2},1)$. We use \eqref{eq:unif-bound-trozos} and Lemma~\ref{lem:bound-diff} and we deduce the estimates
\[
|\mathcal{R}_{1}'(m,n)|\le C (m+1)^{\beta+1/2}(n+1)^{\beta-1/2}\int_{I'_1} (1+x)^\beta\, dx\le \frac{C}{(m+1)^{2}}
\]
and
\[
|\mathcal{R}_{3}'(m,n)|\le C (m+1)^{\alpha+1/2}(n+1)^{\alpha+1/2}\int_{I'_3}(1-x)^{\alpha+1/2}\, dx\le \frac{C}{(m+1)^{2}}.
\]

We analyse now the term $\mathcal{R}_{2}'(m,n)$. By \eqref{eq:R'-size}, using that $\lambda_{n+2}^{(\alpha+1,\beta)}\not= \lambda_m^{(\alpha,\beta)}$ and $\lambda_{n}^{(\alpha+1,\beta)}\not= \lambda_m^{(\alpha,\beta)}$,  it is possible to prove the identity
\begin{multline*}
%\label{eq:R-smoothprime}
\mathcal{R}'_2(m,n)=\frac{S'(m,n+2)-J'(m,n+2)}{\lambda_{n+2}^{(\alpha+1,\beta)}-\lambda_m^{(\alpha,\beta)}}
-\frac{S'(m,n)-J'(m,n)}{\lambda_{n}^{(\alpha+1,\beta)}-\lambda_m^{(\alpha,\beta)}}\\
=\frac{\mathcal{S}'(m,n)-\mathcal{J}'(m,n)}{\lambda_{n+2}^{(\alpha+1,\beta)}-\lambda_m^{(\alpha,\beta)}}
-\frac{2(2n+\alpha+\beta+4)(S'(m,n)-J'(m,n))}
{(\lambda_{n+2}^{(\alpha+1,\beta)}-\lambda_m^{(\alpha,\beta)})(\lambda_{n}^{(\alpha+1,\beta)}-\lambda_m^{(\alpha,\beta)})},
\end{multline*}
where
\begin{equation*}
\mathcal{J}'(m,n)=\int_{I'_2}H_{\alpha,\beta}(x)p_m^{(\alpha,\beta)}(x)
(p_{n+2}^{(\alpha+1,\beta)}(x)-p_n^{(\alpha+1,\beta)}(x))\,d\mu_{\alpha,\beta}(x)
\end{equation*}
and
\[
\mathcal{S}'(m,n)=U_{\alpha+1,\beta}\Big((1-(\cdot))^{-1/2}p_{m}^{(\alpha,\beta)},
p_{n+2}^{(\alpha+1,\beta)}-p_n^{(\alpha+1,\beta)}\Big)(x)\Big|_{x=-1+1/(m+1)^2}^{x=1-1/(m+1)^2}.
\]
By \eqref{eq:Sprim} and \eqref{eq:J'} we obtain that
\begin{equation*}
\left|\frac{2(2n+\alpha+\beta+4)(S'(m,n)-J'(m,n))}
{(\lambda_{n+2}^{(\alpha+1,\beta)}-\lambda_m^{(\alpha,\beta)})(\lambda_{n}^{(\alpha+1,\beta)}-\lambda_m^{(\alpha,\beta)})}\right|\le \frac{C}{|n-m|^2}.
\end{equation*}
Now, from \eqref{eq:unif-bound-trozos}, \eqref{eq:Jaco-der}, and Lemmas \ref{lem:bound-diff} and \ref{lem:bound-diff-der}, we deduce the estimate
\begin{equation*}
|\mathcal{S}'(m,n)| \leq C
\end{equation*}
and therefore
\[
\left|\frac{\mathcal{S}'(m,n)}{\lambda_{n+2}^{(\alpha+1,\beta)}-\lambda_m^{(\alpha,\beta)}}\right|\le \frac{C}{|n-m|^2}.
\]
We deal now with the term $\mathcal{J}'(m,n)$. By using \eqref{eq:L-product} we have that
\begin{multline*}
\lambda_{m}^{(\alpha,\beta)}\mathcal{J}'(m,n)\\
\begin{aligned}
&=
\int_{I'_2}H_{\alpha,\beta}(x)(p_{n+2}^{(\alpha+1,\beta)}(x)-p_n^{(\alpha+1,\beta)}(x))
L^{\alpha,\beta}p_m^{(\alpha,\beta)}(x)\,d\mu_{\alpha,\beta}(x)\\
&=\overline{\mathcal{S}}'(m,n)+\int_{I'_2}L^{\alpha,\beta}(H_{\alpha,\beta}(p_{n+2}^{(\alpha+1,\beta)}-p_n^{(\alpha+1,\beta)}))(x)p_m^{(\alpha,\beta)}(x)\,d\mu_{\alpha,\beta}(x)\\
&=\overline{\mathcal{S}}'(m,n)+\int_{I'_2}H_{\alpha,\beta}(x)L^{\alpha+1,\beta}(p_{n+2}^{(\alpha+1,\beta)}-p_n^{(\alpha+1,\beta)}))(x)
p_m^{(\alpha,\beta)}(x)\,d\mu_{\alpha,\beta}(x)\\&\kern25pt -T'_1(m,n)-T'_2(m,n),
\end{aligned}
\end{multline*}
where
\[
\overline{\mathcal{S}}'(m,n)=U_{\alpha,\beta}\Big(H_{\alpha,\beta}(p_{n+2}^{(\alpha+1,\beta)}-p_n^{(\alpha+1,\beta)}),
p_m^{(\alpha,\beta)}\Big)(x)\Big|_{x=-1+1/(m+1)^2}^{x=1-1/(m+1)^2},
\]
\begin{multline*}
T_1'(m,n)=\int_{I_2'}((1+x)H_{\alpha,\beta}(x)+2(1-x^2)H'_{\alpha,\beta}(x))\\(p_{n+2}^{(\alpha+1,\beta)}-p_n^{(\alpha+1,\beta)})'(x)
p_m^{(\alpha,\beta)}(x)\, d\mu_{\alpha,\beta}(x),
\end{multline*}
and
\begin{multline*}
T_2'(m,n)=\int_{I_2'}((1-x^2)H''_{\alpha,\beta}(x)+(\beta-\alpha-(\alpha+\beta+2)x)H'_{\alpha,\beta}(x))
\\(p_{n+2}^{(\alpha+1,\beta)}(x)-p_n^{(\alpha+1,\beta)}(x))
p_m^{(\alpha,\beta)}(x)\, d\mu_{\alpha,\beta}(x).
\end{multline*}
Applying \eqref{eq_Lpols} we get
\begin{multline*}
\frac{\mathcal{J}'(m,n)}{\lambda_{n+2}^{(\alpha+1,\beta)}-\lambda_m^{(\alpha,\beta)}}\\=
\frac{-\overline{\mathcal{S}}'(m,n)-2(2n+\alpha+\beta+4)J'(m,n)+T'_1(m,n)+T'_2(m,n)}{(\lambda_{n+2}^{(\alpha+1,\beta)}-\lambda_m^{(\alpha,\beta)})^2}.
\end{multline*}
From \eqref{eq:J'}, it is easy to show that
\[
\left|\frac{2(2n+\alpha+\beta+4)J'(m,n)}{(\lambda_{n+2}^{(\alpha+1,\beta)}-\lambda_m^{(\alpha,\beta)})^2}\right|\le \frac{C}{|n-m|^2}.
\]
To estimate the term $\overline{\mathcal{S}}'(m,n)$ we use \eqref{eq:unif-bound-trozos}, \eqref{eq:Jaco-der}, Lemmas \ref{lem:bound-diff} and \ref{lem:bound-diff-der}, and the estimates $|H_{\alpha,\beta}(x)|\leq C(1-x)^{-1/2}$, $|H_{\alpha,\beta}'(x)|\leq C(1-x)^{-3/2}$, for $-1<x<1$. Then,
\begin{equation*}
|\overline{\mathcal{S}}'(m,n)| \leq C (m+1)^{2}
\end{equation*}
and
\[
\left|\frac{\overline{\mathcal{S}}'(m,n)}{(\lambda_{n+2}^{(\alpha+1,\beta)}-\lambda_m^{(\alpha,\beta)})^2}\right|\le \frac{C}{|n-m|^2}.
\]
Finally, we estimate the terms $T_{1}'(m,n)$ and $T_{2}'(m,n)$. We split both of them according to the intervals $V'_1=[-1+1/(m+1)^2,-1+1/(n+1)^2)$, $V'_2=[-1+1/(n+1)^2,1-1/(n+1)^2]$, and $V'_3=(1-1/(n+1)^2,1-1/(m+1)^2]$. Thus, using \eqref{eq:unif-bound-trozos}, Lemma \ref{lem:bound-diff-der}, and the estimate
\[
|(1+x) H_{\alpha,\beta}(x)-2(1-x^2) H'_{\alpha,\beta}(x)|\le C (1+x)(1-x)^{-1/2}, \qquad -1<x<1,
\]
for $m/2\leq n\leq 3m/2$ and $\alpha\ge -1/2$ we have
\begin{align*}
|T'_1(m,n)|&\le C\left((n+1)^{\beta+3/2}\int_{V'_1}(1+x)^{\beta/2+3/4}\, dx\right.
\\&\kern20pt+(n+1)\int_{V'_2}(1+x)^{1/2}(1-x)^{-3/2}\, dx
\\&\kern20pt\left.+(n+1)^{\alpha+5/2}\int_{V'_3}(1-x)^{\alpha/2-3/4}\, dx\right)\le C(m+1)^2.
\end{align*}
Moreover, by \eqref{eq:unif-bound-trozos}, Lemma \ref{lem:bound-diff}, and the estimate
\[
|(1-x^2) H''_{\alpha,\beta}(x)+2(\beta-\alpha-(\alpha+\beta+2)x) H'_{\alpha,\beta}(x)|\le C (1-x)^{-3/2}, \quad -1<x<1,
\]
we conclude that for $m/2\leq n\leq 3m/2$ and $\alpha,\beta \ge -1/2$,
\begin{multline*}
|T'_2(m,n)|\le C\left((n+1)^{\beta-1/2}\int_{V'_1}(1+x)^{\beta/2-1/4}\, dx+\int_{V'_2}(1-x)^{-2}\, dx\right.\\\left.+(n+1)^{\alpha+1/2}\int_{V'_3}(1-x)^{\alpha/2-7/4}\, dx\right)\le C(m+1)^2
\end{multline*}
and the proof of the proposition is finished.
\end{proof}

%%%%%%%%%%%%%%%%%%%%%%%%%%%%%%%%%%%%%%%%%%%%%%%%%%%%%%%%%%%%%%%%%%%%%%%%%%%%%%%%%%%%%%%%%
\section{Proofs of Lemmas \ref{lem:bound-diff} and \ref{lem:bound-diff-der}}
%%%%%%%%%%%%%%%%%%%%%%%%%%%%%%%%%%%%%%%%%%%%%%%%%%%%%%%%%%%%%%%%%%%%%%%%%%%%%%%%%%%%%%%%%
\begin{proof}[Proof of Lemma \ref{lem:bound-diff}]
First of all, note that it is enough to proof that
\begin{equation}
\label{eq:bound-diff-2}
|p_{n+2}^{(a,b)}(x)-p_{n}^{(a,b)}(x)|\le C
\begin{cases}
(n+1)^{a-1/2}, & 1-1/(n+1)^2<x<1,\\
(1-x)^{-a/2+1/4}, & 0\le x\leq1-1/(n+1)^2,
\end{cases}
\end{equation}
because the bound for $-1<x<0$ is obtained immediately from latter by using the relation $P_n^{(a,b)}(-z)=(-1)^nP_n^{(b,a)}(z)$, $-1<z<1$.

It is straightforward to check that
\begin{equation}
\label{eq:lem-1}
p_{n+2}^{(a,b)}(x)-p_n^{(a,b)}(x)=\left(\frac{w_{n+2}^{(a,b)}}{w_n^{(a,b)}}-1\right)p_n^{(a,b)}(x)+w_{n+2}^{(a,b)}(P_{n+2}^{(a,b)}(x)-P_{n}^{(a,b)}(x)).
\end{equation}
From the estimate
\begin{equation*}
%\label{eq:asym-ratio}
\left|\frac{w_{n+2}^{(a,b)}}{w_n^{(a,b)}}-1\right| \leq \frac{C}{n+1}
\end{equation*}
and the uniform estimate \eqref{eq:unif-bound-trozos} (note that if $0\le x<1-1/(n+1)^2$, then $\frac{1}{n+1}\le (1-x)^{1/2} $), we conclude that
\begin{equation}
\label{eq:lem-2}
\left|\left(\frac{w_{n+2}^{(a,b)}}{w_n^{(a,b)}}-1\right)p_n^{(a,b)}(x)\right|
\\\le C\begin{cases}
(n+1)^{a-1/2}, & 1-1/(n+1)^2<x<1,\\
(1-x)^{-a/2+1/4}, & 0\le x\leq1-1/(n+1)^2.
\end{cases}
\end{equation}

Now we apply the following identity (obtained from \cite[18.9.6]{NIST})
\[
-\frac{2n+a+b+2}{2}(1-z)P_n^{(a+1,b)}(z)+aP_n^{(a,b)}(z)=(n+1)(P_{n+1}^{(a,b)}(z)-P_{n}^{(a,b)}(z)),
\]
$0\leq z<1$, $a,b>-1$, to deduce the estimate
\begin{multline*}
w_{n+2}^{(a,b)}|P_{n+1}^{(a,b)}(x)-P_{n}^{(a,b)}(x)|\\
\le \frac{(2n+a+b+2)}{2(n+1)}(1-x)\frac{w_{n+2}^{(a,b)}}{w_{n}^{(a+1,b)}}|p_n^{(a+1,b)}(x)|+\frac{|a|}{n+1}\frac{w_{n+2}^{(a,b)}}{w_{n}^{(a,b)}}|p_n^{(a,b)}(x)|.
\end{multline*}
Therefore, the uniform estimate \eqref{eq:unif-bound-trozos} implies that
\begin{multline}
\label{eq:lem-3}
w_{n+2}^{(a,b)}|P_{n+1}^{(a,b)}(x)-P_{n}^{(a,b)}(x)|\\\le  C\begin{cases}
(n+1)^{a-1/2}, & 1-1/(n+1)^2<x<1,\\
(1-x)^{-a/2+1/4}, & 0\le x\leq1-1/(n+1)^2,
\end{cases}
\end{multline}
and the same bound holds for the term $w_{n+2}^{(a,b)}|P_{n+2}^{(a,b)}(x)-P_{n+1}^{(a,b)}(x)|$.
Then, \eqref{eq:bound-diff-2} follows from \eqref{eq:lem-1}, \eqref{eq:lem-2}, and \eqref{eq:lem-3}.
\end{proof}

\begin{proof}[Proof of Lemma \ref{lem:bound-diff-der}]
First, we assume that $n\neq 0$. By \eqref{eq:Jaco-der}, it is easy to check that
\begin{multline*}
(p_{n+2}^{(a,b)}-p_{n}^{(a,b)})'(x)=\frac{w_{n+2}^{(a,b)}}{w_{n+1}^{(a+1,b+1)}}\frac{n+a+b+3}{2}(p_{n+1}^{(a+1,b+1)}(x)-p_{n-1}^{(a+1,b+1)}(x))\\+
\left(\frac{w_{n+2}^{(a,b)}}{w_{n+1}^{(a+1,b+1)}}\frac{n+a+b+3}{2}-\frac{w_{n}^{(a,b)}}{w_{n-1}^{(a+1,b+1)}}\frac{n+a+b+1}{2}\right)p_{n-1}^{(a+1,b+1)}(x).
\end{multline*}
Then, using that
\[
\left|\frac{w_{n+2}^{(a,b)}}{w_{n+1}^{(a+1,b+1)}}\frac{n+a+b+3}{2}-\frac{w_{n}^{(a,b)}}{w_{n-1}^{(a+1,b+1)}}\frac{n+a+b+1}{2}\right|\leq C,
\]
and the estimate \eqref{eq:unif-bound-trozos} and Lemma \ref{lem:bound-diff}, the result follows. If $n=0$, we proceed in a similar way using \eqref{eq:Jaco-der0} instead of \eqref{eq:Jaco-der}.
\end{proof}

%%%%%%%%%%%%%%%%%%%%%%%%%%%%%%%%%%%%%%%%%%%%%%%%%%%%%%%%%%%%%%%

%%%%%%%%%%%%%%%%%%%%%%%%%%%%%%%%%%%%%%%%%%%%%%%%%%%%%%%%%%%%%%%%%%%%
\end{document}